\documentclass[12pt]{amsart}
\usepackage{txfonts}
\usepackage{amscd,amssymb,mathabx,mathtools}
\usepackage[arrow,matrix,graph,frame,poly,arc,tips]{xy}
\usepackage{graphicx,calrsfs}
\usepackage{tikz}
\usetikzlibrary{decorations.markings}
\usetikzlibrary{shapes}
\usetikzlibrary{backgrounds}
\usetikzlibrary{calc}
\usepackage{mdwlist}
\usepackage{xfrac}

\usepackage{xcolor}
\usepackage{multirow}
\usepackage{enumerate}
\usepackage{verbatim}
\usepackage{comment}
\usepackage{todonotes}

\newcommand\on[1]{\operatorname{#1}}

    \newcommand\ba{\begin{align*}}
    \newcommand\ea{\end{align*}}
    \newcommand\be{\begin{enumerate}}
    \newcommand\ee{\end{enumerate}}
    \newcommand\bp{\begin{proof}}
    \newcommand\ep{\end{proof}}
    \newcommand\bpp{\begin{prop}}
    \newcommand\epp{\end{prop}}
    \newcommand\bpb{\begin{prob}}
    \newcommand\epb{\end{prob}}
    \newcommand\bd{\begin{defn}}
    \newcommand\ed{\end{defn}}
    \newcommand\bh{\begin{hint}}
    \newcommand\eh{\end{hint}}

    \newcommand\bC{\mathbb{C}}
    
    \newcommand\bN{\mathbb{N}}
    \newcommand\N{\mathbb{N}}
    \newcommand\bR{\mathbb{R}}
    
    \newcommand\bZ{\mathbb{Z}}
    \newcommand\Z{\mathbb{Z}}

    \newcommand{\R}[0]{\mathbb{R}}

    \newcommand\supp{\operatorname{supp}}

    \newcommand\gam{\Gamma}
    \newcommand\Mod{\operatorname{Mod}}

    \newcommand\rk{\operatorname{rk}}
    
    \DeclareMathOperator\Homeo{Homeo}
    
    \newcommand\mC{\mathcal{C}}

    \DeclareMathOperator\HS{HS}

    \DeclareMathOperator\ro{RO}

    \DeclareMathOperator\cl{cl}

    \newcommand{\mH}{\mathcal H}
    \newcommand{\mN}{\mathcal N}
     \newcommand{\bSigma}{\mathbf{\Sigma}}
     \newcommand{\bPi}{\mathbf{\Pi}}

    \def\thetitle{{Uniform first order interpretation of the second order theory of countable groups of homeomorphisms}}
    \def\theauthors{{Thomas Koberda, J.~de la Nuez Gonz\'alez}}
    \usepackage{hyperref}
    \hypersetup{
    	colorlinks=false,
    	plainpages,
    	urlcolor=black,
    	linkcolor=black
    	pdftitle=   \thetitle,
    	pdfauthor=  {\theauthors}
    }

    \theoremstyle{theorem}
    \newtheorem{thm}{Theorem}[section]
    
    \newtheorem{lemma}[thm]{Lemma}
    \newtheorem{cor}[thm]{Corollary}
    \newtheorem{prop}[thm]{Proposition}

    \newtheorem*{claim*}{Claim}

    \theoremstyle{remark}

    \theoremstyle{definition}
    \newtheorem{defn}[thm]{Definition}
    \newtheorem{prob}{Problem}[section]

\begin{document}
    \title\thetitle
    \date{\today}
    

    \author[T. Koberda]{Thomas Koberda}
    \address{Department of Mathematics, University of Virginia, Charlottesville, VA 22904-4137, USA}
    \email{thomas.koberda@gmail.com}
    \urladdr{https://sites.google.com/view/koberdat}
    
    \author[J. de la Nuez Gonz\'alez]{J. de la Nuez Gonz\'alez}
    \address{School of Mathematics, Korea Institute for Advanced Study (KIAS), Seoul, 02455, Korea}
    \email{jnuezgonzalez@gmail.com}

    \begin{abstract}
    We show that the first order theory of the homeomorphism group of a compact manifold interprets the full second order theory
    of countable groups of homeomorphisms of the manifold. The interpretation is uniform across manifolds of bounded
    dimension. As a consequence, many classical problems in group theory and geometry (e.g.~the linearity of mapping classes of compact
    $2$--manifolds) are encoded as elementary properties
    of homeomorphism groups of manifolds.
    Furthermore, the homeomorphism group uniformly interprets the Borel
    and projective hierarchies of the homeomorphism group, which gives a characterization
    of definable subsets of the homeomorphism group.
    Finally, we prove analogues of Rice's Theorem from computability theory for
    homeomorphism groups of manifolds. As a consequence, it follows
    that the collection of sentences that isolate the homeomorphism
    group of a particular manifold, or that isolate the homeomorphism groups of manifolds in general, is not definable in second order arithmetic, and that
    membership of particular sentences in these collections cannot be proved in ZFC.
    \end{abstract}
    
    \maketitle

\setcounter{tocdepth}{1}
    

\section{Introduction}

Let $M$ be a compact, connected, topological manifold of positive dimension.
In this paper, we investigate countable subgroups of the group
$\Homeo(M)$ from the point of view of the first order logic of groups, thus continuing a research program initiated together with
Kim~\cite{dlNKK22}. There, we proved that for each compact manifold $M$, there is a sentence in the language of groups which isolates
the group $\Homeo(M)$; that is, there exists a sentence
in the language of group theory that is true in the group of homeomorphisms of an arbitrary compact manifold $N$ if and only if $N$
is homeomorphic to $M$.

The overarching theme of this paper is that the first order theory of
$\Homeo(M)$ is expressive enough to interpret arbitrary sequences of elements of $\Homeo(M)$.
More concretely:
on the one hand, the question of determining the isomorphism type of the subgroup of
$\Homeo(M)$ generated by a finite list of elements is difficult, and in general is intractable.
On the other hand, it can be shown by general Baire category arguments (Proposition 4.5 in~\cite{Ghys2001}, cf.~Chapter 3 in~\cite{KKM2019})
that generically, pairs of homeomorphisms will generate
nonabelian free groups. Even in the case of one-dimensional manifolds, general finitely generated groups of homeomorphisms
(and even diffeomorphisms)
can be extremely complicated;
cf.~\cite{BKK2019JEMS,KK2018JT,KKL2019ASENS,KL2017,KK2020crit,KK2021book,BMRT,KKR2020}.

Since $\Homeo(M)$ can interpret arbitrary sequences of elements in the underlying group, the first order theory of
$\Homeo(M)$ is expressive enough to decide if a countable subgroup if isomorphic to a given
finitely presented group;
as another example, by identifying tuples of homeomorphisms which generate a particular isomorphism type of groups (e.g.~a free
group of rank two), we obtain
an upper bound on the complexity of the set of tuples which generate that type of group.
Thus, the elementary theory of the homeomorphism group $\Homeo(M)$ encodes a substantial
amount of the algebraic structure of this group.

\subsection{Main results}
All results stated in this section hold for arbitrary compact, connected manifolds; we assume connectedness mostly for convenience. There is a dependence of the
formulae on the dimension of the
underlying manifold, but otherwise all formulae are uniform across manifolds of 
fixed dimension.
Throughout, we let \[\Homeo_0(M)\leq \mH\leq\Homeo(M),\]
where here $\Homeo_0(M)$ denotes the identity component of $\Homeo(M)$. Unless otherwise noted, formulae are uniform in $\mH$, which
is to say they do not depend on which subgroup between $\Homeo_0(M)$ and $\Homeo(M)$ we consider.
We suppress $M$ from the notation $\mH$ since it will not cause confusion.

To begin, $\mH$ is viewed as a structure in the language of group theory. The content
of the paper~\cite{dlNKK22} is that the language of group theory in $\mH$ admits a
conservative expansion wherein many more things can be interpreted: specifically,
the sorts of regular open sets $\ro(M)$ in $M$, the natural numbers $\N$, the real numbers
$\R$,
and points in $M$ can be parameter-free interpreted. Moreover, natural predicates, both
internal to these sorts (e.g.~arithmetic) and
relating these sorts to each other, are uniformly definable;
see Theorem~\ref{thm:kkdlng} below.

The main result of this paper is the conservative interpretation of a sequence of new sorts
in $\mH$, which are written $\HS_i(M)$ for $i\geq 0$. The meanings of these sorts are as
follows:
\begin{itemize}
    \item The elements of $\HS_0(M)$ are in canonical correspondence with homeomorphisms of
    $M$.
    \item For $i\geq 1$, the elements of $\HS_i(M)$ are in canonical correspondence with
    sequences of elements in $\HS_{i-1}(M)$.
    \item These sorts admit parameter-free definable predicates for manipulating them and
    for relating them to each other and to the home sort.
\end{itemize}

We call $\HS(M)$, the union of the sorts $\{\HS_i(M)\}_{i\in\N}$,
\emph{hereditarily sequential
subsets of $\Homeo(M)$}; this is by analogy to (and by generalization of)
\emph{hereditarily finite sets} (cf.~Section 3 in~\cite{KMS-2021},
for instance).

Note that for $n\geq 2$, elements of $\HS_n(M)$ are not really subsets of $\Homeo(M)$.
One would be justified in calling an interpretation of $\HS_1(M)$ \emph{countable second
order logic}, since then one can quantify freely over countable subsets of $\Homeo(M)$.
Then, for $n\geq 2$ one would be justified in calling an interpretation of $\HS_n(M)$
\emph{countable $(n+1)^{st}$ order logic}. The distinction between countable second order
logic and countable higher order logics collapses in our situation; this is because
our interpretation of countable second order logic (i.e.~$\HS_1(M)$)
encodes countable sequences via fixed
length definable tuples, up to a definable equivalence relation. Thus for all
$n\geq 2$, an interpretation 
of $\HS_n(M)$ would consist of sequences of fixed length finite tuples, which themselves
would be encoded by fixed length finite tuples in $\mH$.

Hereditarily sequential sets subsume hereditarily finite sets via
a straightforward padding construction.

\begin{thm}\label{thm:hered-seq-homeo}
Let $D\geq 1$ be a natural number, and let \[\Homeo_0(M)\leq \mH\leq \Homeo(M).\]
Then there is a conservative expansion of the language of group theory and a uniform
interpretation of the union of the sorts
$\HS(M)$ in $\mH$ that is valid for all manifolds $M$ with $\dim M\leq D$. The elements in
the sort
$\HS_0(M)$ canonically correspond to elements of $\Homeo(M)$.

Moreover, the following predicates are definable without parameters:
\begin{enumerate}
\item
For each $i$ and each $j\in\bN$, the $j^{th}$ element $s(j)$ of a sequence
$s\in\HS_i(M)$;
\item
For each $i\geq 0$, a membership predicate \[\in_i\subseteq \HS_i(M)\times\HS_{i+1}(M)\]
defined recursively by:
\begin{enumerate}
    \item $\gam\in_0 s$ if and only if there is a $j$ such that $\gam=s(j)$.
    \item $s\in_i t$ if and only if there is a $j$ such that $s=t(j)$.
\end{enumerate}
\item
Member-wise group multiplication within $\HS_1(M)$, i.e.~a predicate 
$\on{mult}_{i,j,k}(\sigma)$ such that for all sequences $s\in\HS_1(M)$, we have
$\on{mult}_{i,j,k}(s)$ if and only if $s(i)\cdot s(j)=s(k)$.
\item\label{i:membership}
Membership of an element in $\HS_0(M)$ in $\mH$, i.e.~a predicate \[R\subseteq \mH\times
\HS_0(M)\] such that $(g,\gam)\in R$ if and only if $\gam$ canonically encodes $g$.
\item 
The extended support $\supp^e f$ of an element $f\in\HS_0(M)$, i.e.~a predicate
\[\supp^e\subseteq \HS_0(M)\times\ro(M)\] such that $(\gam,U)\in\supp^e$ if and only if
the homeomorphism encoded by $\gam$ has extended support equal to $U$.
\end{enumerate}
\end{thm}

We will sometimes abuse notation and suppress the subscript in
$\in_i$ when no confusion can
occur.
We note that Item~\ref{i:membership} is crucial and what makes Theorem~\ref{thm:hered-seq-homeo} not a consequence of~\cite{dlNKK22}. Moreover, Item~\ref{i:membership} will allow
us to characterize definable sets in $\mH$ below (see Theorem~\ref{thm:descriptive}).

The key step in interpreting $\HS(M)$ yields the following, which is of independent
interest. See Lemma~\ref{lem:point-seq}.

\begin{prop}\label{prop:point-seq}
    For manifolds of fixed dimension, the group $\mH$ admits a uniform, parameter-free
    interpretation of the sort $\on{seq}(M)$ of countable
    sequences of points in $M$, which is uniform for all
    manifolds of dimension $d$.
    Moreover, the predicate $p\in\sigma$ expressing membership
    of a point $p$ in a sequence $\sigma$, and the predicate $\sigma(i)=p$ expressing that
    $p$ is the $i^{th}$ term of $\sigma$, are both parameter-free definable.
\end{prop}

The interpretability of hereditarily sequential sets in $\mH$ has a large number of consequences
with regard to definability in $\mH$.

\begin{prop}\label{prop:unif-intermediate}
The class $\mathcal K$ of subgroups of $\Homeo(M)$ that contain $\Homeo_0(M)$
is uniformly interpretable (with parameters) in $\mH$, as definable subsets of the sort
$\HS_0(M)$.
Among the elements of $\mathcal K$ are three
canonical parameter-free interpretable subgroups, namely \[\{\Homeo_0(M),\Homeo(M), \mH\}.\]
\end{prop}

Combining Theorem~\ref{thm:hered-seq-homeo} and Proposition~\ref{prop:unif-intermediate}, we
will be able to interpret hereditarily sequential
sets in other groups lying between $\Homeo_0(M)$ and $\Homeo(M)$, and in various parameter--free interpretable quotients such as the topological mapping class group
$\Mod(M):=\Homeo(M)/\Homeo_0(M)$.

\subsection{Group theoretic consequences of the main results}

Theorem~\ref{thm:hered-seq-homeo} immediately implies that within the first order theory of $\mH$, we have unfettered access
to the full second order theory of countable subgroups of $\Homeo(M)$; in particular, we may freely quantify over countable subgroups,
as well as their subgroups, and homomorphisms between them. Since $\mH$ also interprets second order arithmetic, we may
uniformly interpret combinatorial (and even analytic) group theory within the first order theory of $\mH$; that is, we can encode arbitrary recursively
presented groups within second order arithmetic, and we may also manipulate them (i.e.~test for nontriviality of words, solve the conjugacy
problem, test for isomorphism, determining if a subgroup has finite index, measure the index of a finite index subgroup, test for amenability, test Kazhdan's property (T), etc.; the reader is
directed to~\cite{simpson-book} for an extensive
discussion of mathematics that can be developed within second order arithmetic). 
Observe that an abstract countable group will generally have to be specified with parameters, in the form of a sequence of
natural numbers.

For abstract finitely generated groups, the standard concepts from geometric group theory can also be interpreted, such as the
Cayley graph with respect to a finite generating set, growth,
hyperbolicity, and quasi-isometry.

Below, we give a (non-exhaustive)
list some concepts that can be encoded within the elementary theory of $\mH$.

\begin{thm}\label{thm:gp-thy-conseq}
The following group-theoretic sorts and predicates are parameter-free
interpretable in $\mH$, uniformly for all compact manifolds $M$ of fixed dimension.
\begin{enumerate}
\item
Countable subgroups of $\Homeo(M)$ and their full second order theory.
\item
The topological mapping class group $\Mod(M)$ of $M$,
i.e.~the group \[\pi_0(\Homeo(M))=\Homeo(M)/\Homeo_0(M),\]
and the full second order theory
of $\Mod(M)$.
\item
For a sequence $\underline g$ of homeomorphisms or mapping classes, the membership
predicate for the subgroup $\langle\underline g\rangle$.
\item
Finite generation and finite presentability of arbitrary countable subgroups of $\Homeo(M)$ or $\Mod(M)$.
\item
Residual finiteness of arbitrary countable subgroups of $\Homeo(M)$ and $\Mod(M)$.
\item
Linearity of arbitrary countable subgroups of $\Homeo(M)$ and $\Mod(M)$, i.e.~a predicate which holds if and only if the corresponding
group is linear over a field of characteristic zero.
\item
A predicate expressing isomorphism with a particular group that is parameter-free definable in
second order arithmetic
(e.g.~isomorphism with some finite index subgroup of $\on{SL}_n(\bZ)$).
\item For a finitely generated subgroup of $\Homeo(M)$ or $\Mod(M)$, a predicate expressing
whether this group
is amenable of has Kazhdan's Property (T).
\end{enumerate}
\end{thm}

Thus, the first order theory of $\mH$ encodes many well-known conjectures
as elementary properties of homeomorphism groups. These include the linearity of mapping
class groups of compact $2$--manifolds (see~\cite{FM2012} for a general reference, and 
Question 1.1 of~\cite{Margalit-2019}), property (T) for mapping class groups of compact $2$--manifolds,
finite presentability of the Torelli group of a compact $2$--manifold (see~\cite{Putman2009,MP2008}, and especially Section 5 of~\cite{Margalit-2019})
the existence of an infinite, discrete, property (T) group of homeomorphisms of the 
circle (see~\cite{Duchesne23,Navas2002ASENS,BFGM2007AM}, and especially Question 2 of
~\cite{Navas-icm}), the amenability of Thompson's
group $F$~\cite{CFP1996,BurilloBook}, and many cases of the Zimmer program (i.e.~faithful continuous
actions of finite index subgroups of lattices in semisimple Lie groups
on compact manifolds~\cite{Fisher2011,Fisher2020,Fisher2020a,BFH2020,BFH2022}).

\subsection{Descriptive set theory}\label{ss:descriptive}

Much of the foregoing discussion treats $\Homeo(M)$ as a discrete group.
We wish to observe further that the first order theory of $\mH$ recovers the topology of $\Homeo(M)$,
and in fact the full projective hierarchy of subsets of $\Homeo(M)$.
More precisely,
we have the following.

\begin{thm}\label{thm:descriptive}
The following sorts are uniformly interpretable in $\mH$, viewed as a subset of $\HS_0(M)$,
uniformly in manifolds of fixed dimension:
\begin{enumerate}
\item
Open and closed sets in $\Homeo(M)$.
\item
Borel sets in $\Homeo(M)$, and the full Borel hierarchy of $\Homeo(M)$.
\item
The projective hierarchy in $\Homeo(M)$.
\end{enumerate}
The membership predicate $\in$ is parameter-free interpretable for sets in these sorts.

Moreover, the topology of $\mH$, as well as the Borel hierarchy
and projective hierarchy of $\mH$ are all uniformly definable among manifolds of
bounded dimension.
\end{thm}

As a consequence, we will obtain
the following general fact about definable subsets of $\Homeo(M)$:

\begin{thm}\label{thm:definable}
    A set is definable (with parameters) in $\mH$ if and only if it lies in the projective
    hierarchy.
\end{thm}

\subsection{Undefinability  and independence}

As is implicit from the
uniform parameter--free interpretation of second order arithmetic in $\mH$ as produced in~\cite{dlNKK22}, not only is the 
first order theory
of $\mH$ (and of $\Homeo(M)$ in particular)
undecidable, but in fact there are elementary properties of homeomorphism groups of manifolds whose validity is independent of
ZFC. A question therefore is whether
or not there are ``natural" first order group theoretic statements in $\mH$ that are independent of ZFC, and this is unclear to the authors.

There are also
many natural undefinable sets in arithmetic which are directly related to 
compact manifolds and their homeomorphism groups, which
we record here. Manifolds and their homeomorphism groups can be
formalized in second order arithmetic; however, there is some sense in which the manifold homeomorphism group
recognition problem is at least
as complicated as full true second order arithmetic,
which we now make precise.

Choosing
a numbering of the language of groups, we obtain a G\"odel numbering of sentences in group theory. For a fixed compact manifold
$M$, one can consider the set of sentences in group theory (viewed as a subset of $\bN$ via their G\"odel numbers) which isolate
$\Homeo(M)$. Similarly, one may consider the set of sentences in group theory which isolate some isomorphism
type of compact manifold homeomorphism group. It turns out that neither of these sets is definable in arithmetic. For a sentence $\psi$, we
write $\#\psi$ for its G\"odel number with respect to a fixed numbering of the language.

\begin{thm}\label{thm:undefinable}
Let $M$ be a fixed compact manifold and let $N$ be an arbitrary compact manifold.
\begin{enumerate}
\item
The set
\[\on{Sent}_M:=\{\#\psi\mid (\Homeo(N)\models \psi) \longleftrightarrow (M\cong N)\}\] is not definable in second order arithmetic.
\item
The set \[\on{Sent}:=\{\#\psi\mid \textrm{$\psi$ isolates $\Homeo(N)$ for some compact manifold $N$}\}\] is not definable in 
second order arithmetic.
\end{enumerate}
In particular, these sets
are not decidable.
\end{thm}

In Theorem~\ref{thm:undefinable}, the group $\Homeo$ can be replaced by any group lying between $\Homeo_0$ and $\Homeo$.
We will show in Section~\ref{sec:undefinable} that membership of G\"odel numbers in $\on{Sent}_M$ or $\on{Sent}$ cannot be proved
within ZFC.

More generally than Theorem~\ref{thm:undefinable}, we will prove that for any
class $\mathcal M$ of compact manifold homeomorphism groups which is isolated by a
single sentence,
the set of G\"odel numbers of sentences isolating $\mathcal M$ is undefinable in second
order arithmetic; this gives an analogue of Rice's Theorem (i.e.~nontrivial classes of
partially recursive functions are not computable) for
homeomorphism groups of manifolds. In fact, we will prove that
if $\mathcal F$ consists of nonempty sets of homeomorphism groups of compact manifolds
which are
isolated by first order sentences, and if
$A\subseteq F$ is proper, then the set of G\"odel numbers of sentences isolating
elements of $A$ is not definable in second order arithmetic. See Theorem~\ref{thm:rice}
and Theorem~\ref{thm:rice-strong} for precise statements.

\subsection{Organization of the paper}

In Section~\ref{sec:background}, we gather preliminary material about topological manifolds and the first order theory of homeomorphism
groups of manifolds. Section~\ref{sec:hereditary} proves Theorem~\ref{thm:hered-seq-homeo}, the main result of the paper.
Section~\ref{sec:mcg} interprets mapping class groups of manifolds as well as intermediate subgroups lying between
$\Homeo_0$ and $\Homeo$ of manifolds, and discusses Theorem~\ref{thm:gp-thy-conseq}. Section~\ref{sec:descriptive} discusses descriptive set theory and the projective hierarchy
in $\Homeo(M)$.
 Section~\ref{sec:undefinable} proves Theorem~\ref{thm:undefinable} and the analogues of
 Rice's Theorem.

Throughout this paper, we have tried to balance mathematical precision with clarity. To give completely precise and explicit
formulae is possible, though
extremely unwieldy and unlikely to yield deeper insight. Thus, we have often
avoided giving explicit formulae, either explaining
how to obtain them in English with enough precision that the formulae could be produced if desired, or we have avoided them entirely
when certain predicates are obviously
definable in second order arithmetic or in the countable second order theory of a group.

\section{Background}\label{sec:background}
We first gather some preliminary results. Throughout, we will always assume that all manifolds are compact, connected, and second countable.

\subsection{Results from geometric topology of manifolds}
We will appeal to the following fact about compact topological manifolds. We write $B(i)\subset\bR^d$ for the
closed ball of radius
$i$ about the origin. We write $H(i)\subset\bR^d_{\geq 0}$ for the half-ball of radius $i$ about the origin in the half-space $\bR^d_{\geq 0}$.
That is, $H(i)=B(i)\cap \bR^d_{\geq 0}$.
A \emph{collared ball} in a $d$--dimensional
manifold $M$ is a map \[B(1)\longrightarrow M\] which is a homeomorphism onto its image, and
which extends to a homeomorphism of $B(2)$ onto its image, and a \emph{collared half-ball} in a manifold with boundary
is defined analogously in the usual sense, so that the image of the origin in $\R^d$
lands in the boundary $\partial M\subseteq M$
and the intersection of the image of $H(i)$ with 
$\partial M$ is a collared open ball in $\partial M$.

An open set in $M$ is \emph{regular} if it is equal to the interior of its closure.
We will say that a regular open set $U$ is a \emph{regular open collared ball} if it is the interior
of a collared open ball. A \emph{regular open collared half-ball} is a regular open set that is the interior of a collared half-ball. A
regular open collared half-ball meets the boundary of $M$ in a regular open collared ball.

\begin{prop}[See Chapter 3 in~\cite{coornaert2015}, Theorem IV.2 in~\cite{HW1941}, Theorem 3 in~\cite{ostrand71}, Section 6.1 in~\cite{dlNKK22}]\label{prop:bounded cover}
Let $M$ be a compact, connected manifold of dimension $d$. Then there exists a computable function $n(d)$ such that the following
conclusions hold.
\begin{enumerate}
\item
If $M$ is a closed topological manifold then there exist
$n(d)$ collections of disjoint collared balls $\{B_1,\ldots,B_{n(d)}\}$ such that \[M=\bigcup_{i=1}^{n(d)} B_i.\]
\item
If $\partial M\neq\varnothing$ then the following conclusions hold.
\begin{enumerate}
\item
For every collar neighborhood $U\supseteq \partial M$, there exist collections
of disjoint collared balls $\{B_1,\ldots,B_{n(d)}\}$
and collections of disjoint collared half-balls $\{H_{1},\ldots,H_{n(d-1)}\}$
such that \[M\setminus U\subseteq \bigcup_{i=1}^{n(d)} B_i\] and such that
\[\cl U\subseteq \bigcup_{j=1}^{n(d-1)} H_{j}.\]
\end{enumerate}
\end{enumerate}
\end{prop}

In Proposition~\ref{prop:bounded cover}, note that each $B_i$ and each
$H_i$ is a (possibly
disconnected) set, each component of which a collared ball or collared half-ball,
respectively.

\begin{proof}[Proof of Proposition~\ref{prop:bounded cover}]
We will assume that $M$ is closed; the argument for manifolds with boundary is a minor variation on the proof given here.

This essentially follows from the fact that $M$ can be embedded in $\bR^{2d+1}$. Choose such an embedding, which by scaling we may assume
lies in the unit cube $I^{2d+1}$. For any positive threshold $\epsilon>0$, we may cover $I^{2d+1}$ by $2d+2$ collections of regular open
sets $\{B_1,\ldots,B_{2d+2}\}$, each consisting of disjoint collared open Euclidean balls, with each ball having diameter at most $\epsilon$.
Moreover, we may assume that any two components of any $B_i$ are separated by a distance that is uniformly bounded away from zero. These
claims follow from standard constructions in Lebesgue covering dimension; see Chapter 3 in~\cite{coornaert2015}, Chapter 50 in~\cite{Munkres-book}.

Choose an atlas for $M$ such that for an arbitrary component $V$ of some $B_i$,
we have that the intersection $V\cap M$ lies in a coordinate
chart. This can be achieved by setting $\epsilon$ small enough with respect to a fixed
atlas for $M$, as follows from the Lebesgue Covering Lemma.

Let $U\cong\bR^d$ be such a coordinate chart of $M$ and let $B=B_i$ for some $i$.
Then, $U\cap B$ is a collection of open sets which are
separated by a definite distance $\delta>0$ which is independent of $U$.
For any component $V\in\pi_0(B)$ such that $V\cap M$ is entirely contained in $U$,
we may cover
$V\cap M$ with collared open balls (in $M$) which are contained in a $\delta/3$ neighborhood of the closure of $V$ in $M$.
This covering may
be further refined to be a covering by regular collared balls having order at most $d+2$;
in particular, the closure of $V$ is covered by at most $d+2$ collections of regular
open sets, whose components consist of disjoint collared open balls. Repeating this construction for each component $V\in\pi_0(B)$,
we obtain a collection of $d+1$ regular open sets whose components are collared open balls that cover $B\cap M$. Allowing
$B$ to range over $\{B_1,\ldots,B_{2d+2}\}$, we obtain $(d+1)(2d+2)$ regular open sets covering $M$, all of whose components are
collared open balls, as desired.
\end{proof}

The importance of Proposition~\ref{prop:bounded cover} is that many of the formulae we build in this paper will be uniform in the underlying
manifold, provided that the dimension is bounded. This is reflected in the dependence
of $n(d)$ on $d$. The proof of the following corollary
is straightforward, and we omit it.

\begin{cor}\label{cor:single ball}
Let $M$ be a compact, connected manifold of dimension $d$, and let $n(d)$ be as in
Proposition~\ref{prop:bounded cover}.
\begin{enumerate}
\item
If $M$ is closed then $M$ can be covered by $n(d)$ regular open collared balls.
\item
If $\partial M\neq \varnothing$ and if $N$ is a component of $\partial M$, then there is a tubular neighborhood of $N$ whose
closure can be covered
by $n(d-1)$ regular open collared half-balls. Moreover, for all tubular neighborhoods $U\supseteq\partial M$, we have $M\setminus U$
can be covered by $n(d)$ regular open collared balls.
\end{enumerate}
\end{cor}

\subsection{Results about the first order theory of homeomorphism groups of manifolds}
The present paper builds on the results of the authors' joint paper with Kim~\cite{dlNKK22}.
In that paper, we investigated the first order theory of $\Homeo(M)$ for a compact manifold
$M$, and in particular proved that each group $\Homeo(M)$ is quasi-finitely axiomatizable
within the class of homeomorphism groups of manifolds.

The central result of this paper is the interpretation of $\HS(M)$, which does not follow
from the paper~\cite{dlNKK22}. However, we shall require tools which were developed in that
paper in order to prove the results in this paper. We will briefly list the relevant results
that we use here. In the following theorem, if $U\subseteq M$ is an open set and
$G\leq\Homeo(M)$, then we write $G[U]$ for the \emph{rigid stabilizer} of $U$, consisting
of all elements of $G$ which are the identity outside of $U$.

The following result follows from the fact that $\mH$ conservatively interprets, without
parameters, a structure called AGAPE; see Section 3 of~\cite{dlNKK22}. We have given
more precise citations for most enumerated
statements that refer to ~\cite{dlNKK22}. The statements below differ slightly from the way they are stated in
~\cite{dlNKK22} in order to better serve our purposes,
though there is no difference in content.

\begin{thm}[See~\cite{dlNKK22}]\label{thm:kkdlng}
Let $M$ be a compact, connected, topological manifold of dimension at least one,
and let \[\Homeo_0(M)\leq \mH\leq \Homeo(M).\] Then there exists a sentence
$\psi_M$ in the language of group theory such that for all compact manifolds $N$ and all subgroups
\[\Homeo_0(N)\leq\mH'\leq \Homeo(N),\] we have $\mH'\models \psi_M$ if and only if $N\cong M$. Moreover,
the following sorts and predicates are interpretable without parameters in $\mH$, uniformly in $M$.
\begin{enumerate}
\item
The Boolean algebra $\ro(M)$ of regular open sets of $M$, equipped with an action of $\mH$;
that is, a predicate \[\on{Act}\subseteq \mH\times \ro(M)\times\ro(M)\] such that $(g,U,V)\in
\on{Act}$ if and only if $g(U)=V$ in $M$; the interpretation of $\ro(M)$ is uniform for
all manifolds, including noncompact ones. See Section 2.2 and Theorem 3.4.
\item Predicates expressing connectedness of regular open sets, as well as that a regular
open set $U$ is a connected component of a regular open set $V$.
See Lemma 3.6 and Corollary 3.7.
\item\label{rcb1} A predicate $\on{RCB}\subseteq\ro(M)$ such that $U\in\on{RCB}$ if and only if
the closure of $U$ lies in a collared open ball in $M$. See Lemma 3.10.
\item\label{rcb2} A predicate $\on{RCB}^{\partial}\subseteq\ro(M)$
such that $U\in\on{RCB}^{\partial}$ if and only if
the closure of $U$ lies in a collared open half-ball in $M$.
\item\label{item:components}
Second order arithmetic $(\N,0,+,\times,<,\subset)$, and a definable predicate
\[\#\subseteq \N\times\ro(M)\] such that $(n,U)\in\#$ if and only if $U$ has exactly $n$
components;
moreover, if $\varnothing\neq U\in\ro(M)$, then
second order arithmetic can be interpreted using only $U$ and $\mH[U]$. See Section 4.
\item\label{c:points}
Points $\mathcal P(M)$ of $M$, and more generally finite tuples $\mathcal P^{<\infty}(M)$
of points in $M$; moreover, a predicate $\in_{\mathcal P}\subseteq \mathcal P(M)\times \ro(M)$ such
that $(p,U)$ lies in $\in_{\mathcal P}$ if and only if the statement 
$p\in U$ is true in $M$. See Section 5.
\item Predicates expressing that a point of $M$ belongs to a union of two regular open sets,
and that a point belongs to the closure of a regular open set. See Section 5.
\item\label{c:cover} For each $n$, predicate expressing that a collection of $n$ regular open sets covers
the closure of a regular open set $U$.
\item 
Exponentiation, i.e.~a definable function \[\exp\colon \mH\times\Z\times M\longrightarrow 
M\]
with the property that \[\exp(g,n,p)=
g^n(p) \quad\textrm{in $M$}.\] See Section 5.3.
\item A predicate which holds for a regular open set $U$ if and only if $U$
contains a tubular neighborhood of $\partial M$ in $M$. See Theorem 7.1.
\end{enumerate}
\end{thm}

In view of Theorem~\ref{thm:kkdlng}, we will assume that $\mH$ is implicitly equipped
with the sorts of regular open sets of $M$, second order arithmetic, and points, as
well as the relevant predicates listed in the theorem.

Some items in Theorem~\ref{thm:kkdlng} require special comment.
Item~\ref{rcb1} was only formally proved for manifolds of dimension 2 or higher, though
for manifolds of dimension one, the proof is even easier. By the characterization of connected
sets in one-manifolds, it suffices to express that $U$ is contained in a connected regular
open set $V$, and that there is a homeomorphism $h$ of $M$ such that $V\cap h(V)=\varnothing$.

Item~\ref{rcb2} was not formally stated in~\cite{dlNKK22}, though it is not difficult
to find such a formula. One expresses that a regular open set $U$ accumulates on a single
component $N$ of $\partial M$, as is easily deduced from 3.4.3. One then requires the
existence of a homeomorphism $h$ fixing each component of the boundary of $M$, which
moves $U$ into an arbitrary half-ball in $N$; half-balls are interpreted explicitly in
Section 7 of~\cite{dlNKK22}.

In item~\ref{c:points},
a point $p\in M$ is encoded by an equivalence class of regular open sets, up to
definable equivalence. If $U\subseteq M$ is a regular open set and $p\in U$ then there is
a regular open set $V\subseteq U$ which encodes or \emph{isolates} $p$; this is implicit in Section 5 of
~\cite{dlNKK22}. In particular, if $U$ is a regular open set with infinitely many components
$\{U_i\}_{i\in\N}$ and if $p_i\in U_i$ is a point for each $i$, then the set of points
$\bigcup_i p_i$ is encoded by a single regular open set $V$, which has the property that
$V\subseteq \bigcup_i U_i$ and such that $V\cap U_i$ encodes the point $p_i$. We will
abbreviate the predicate $\in_{\mathcal P}$ by $\in$.

Observe that the exponentiation function, together with the membership predicate relating
$\mathcal P(M)$ to $\ro(M)$ allows one to express that $g^n(U)=V$ for group elements in $\mH$,
integer exponents, and pairs of regular open sets, since we may express that
\[\exp(g,n,p)\in V \leftrightarrow p\in U.\]

The sentence $\psi_M$ in Theorem~\ref{thm:kkdlng} is said to \emph{isolate} $M$
(or its homeomorphism group).
We note that in~\cite{dlNKK22}, the proof of the content of
Theorem~\ref{thm:kkdlng} was given for manifolds of dimension at least two. This was done
purely to simplify some of the arguments and shorten the exposition; the proofs themselves can easily be generalized to manifolds of dimension
one.

We note that even though we will refer to collared 
balls and half-balls in the sequel,
these are concepts in the metalanguage; we will never appeal to these objects directly
in the formal language.

To make one further observation about the relationship between $\Homeo(M)$, its countable
subgroups, and arithmetic, we remark the following: $\Homeo(M)$ clearly contains
many countable subgroups that are definable in arithmetic, including cyclic groups and
free groups. Some subgroups of $\Homeo(M)$ are in fact bi-interpretable with
first order arithmetic, such as Thompson's groups $F$ and $T$ by~\cite{lasserre}; it is not difficult to
show that $F$ in fact arises as a subgroup of $\Homeo(M)$ for all positive dimensional
manifolds. Most countable subgroups of $\Homeo(M)$ are not definable in first order
arithmetic, simply because $\Homeo(M)$ interprets second order arithmetic. Indeed, then
any countable elementary subgroup of $\Homeo(M)$ (which exists by the L\"owenheim--Skolem
Theorem) has too complicated a theory to be interpretable in arithmetic. The reader
may find a more detailed discussion in the authors' paper~\cite{KdlN24}.

\section{Hereditarily sequential sets of homeomorphisms of a manifold}\label{sec:hereditary}

Let $M$ and $\mH$ be as above and fixed, and fix the notation $d\geq 1$ for the dimension of $M$.
In this section, we prove Theorem~\ref{thm:hered-seq-homeo}; the uniformity of the
interpretation among manifolds of a fixed dimension will be clear, and by taking disjunctions
we obtain an interpretation that is valid for all manifolds of dimension bounded by
a prescribed constant D. We prove
the result in several steps.

\subsection{Interpreting $\mathrm{HS}_0(M)$}
We begin by interpreting the sort $\HS_0(M)$ in $\mH$, and show that its members
canonically correspond to elements of $\Homeo(M)$. This itself is done in several steps.
The reader should remember for the duration of the proof that we are encoding a
homeomorphism of $M$ by a proxy for its graph; the reader may pretend $M$ is closed
on a first reading, for simplicity.

The scheme for finding parameter-free interpretations of new sorts in $\mH$ will follow
the basic scheme:
\begin{enumerate}
    \item Encode data describing the new sort within various sorts of
    topological data to which we have
    access in view of Theorem~\ref{thm:kkdlng}; oftentimes this data requires making
    choices, which amounts to an interpretation with parameters.
    \item Observe that the set of suitable parameters is itself parameter-free definable
    within the relevant sort.
    \item Eliminate parameters by quantifying over the relevant space of parameters.
\end{enumerate}

The basic idea to interpret $\HS_0(M)$ is to fix a finite cover of $M$,
move the charts in the cover to a single
chart in $M$ (forming a finite set of ``pages"), and then taking countably many disjoint
copies
of these pages. In each copy, we choose a point, which gives us the intermediate
result of being able to interpret the sort of countable sequences of points in
$M$; since points in $M$ are encoded by equivalence
classes of regular open sets in $M$ wherein only the local structure of the open set
near the point being encoded matters, we may encode the countable sequence of points by
a single suitable equivalence class of regular open sets. By considering a sequence
$\sigma$ of points in $M$, we may consider the odd and even index points in $\sigma$, thus
obtaining a countable collection of points in
$M\times M$. We then place definable conditions on such pairs to make sure the points
occurring in each coordinate are dense in $M$, and so that these pairs actually arise from
the graph of a homeomorphism of $M$. We have included some figures to aid the reader.

\begin{lemma}\label{lem:point-seq}
    The group $\mH$ admits a parameter-free interpretation of the sort $\on{seq}(M)$ of
    countable sequences of points in $M$, which is uniform for all manifolds of
    dimension $d$. Moreover, the predicate $p\in\sigma$ expressing membership
    of a point $p$ in a sequence $\sigma$, and the predicate $\sigma(i)=p$ expressing that
    $p$ is the $i^{th}$ term of $\sigma$, are both parameter-free definable.
\end{lemma}

For technical reasons, we first prove the lemma in the case where $M$ is not the interval,
and give an adapted proof for the interval later.

\begin{proof}[Proof of Lemma~\ref{lem:point-seq} for $M\neq I$]
   We retain the notation $n(d)$ from
Proposition~\ref{prop:bounded cover}. 

{\bf Choose a cover $M$.} We first fix a collection of regular open sets in $M$ of
bounded cardinality (depending on $d$) which cover $M$, and which can be used as charts in
an atlas
for $M$. Fix a collar neighborhood $K$ of $\partial M$ in $M$. Since $M$ has dimension $d$
and $\partial M$ has dimension $d-1$, Proposition~\ref{prop:bounded cover} shows that
$M\setminus K$ can be covered by $n(d)$ regular open sets, each component of which is a
collared open ball, and each component of $K$ can be cover by $n(d-1)$ such sets consisting
of collared open half-balls. By Items~\ref{rcb1}, \ref{rcb2}, and~\ref{c:cover} of Theorem~\ref{thm:kkdlng}, we
may express the existence of collections
\[\mathfrak B=\{U_1,\ldots,U_{n(d)}\}\quad\textrm{and} \quad \mathfrak H =\{V_1,\ldots,V_{n(d-1)}\}\] such that:
\begin{enumerate}
    \item The sets $\mathfrak B\cup\mathfrak H$ cover $M$; this is expressible since
    we simply require every point of $M$ to lie in an element of
    $\mathfrak B\cup\mathfrak H$.
    \item For $W\in \mathfrak B\cup\mathfrak H$ and $W_0$ a component of $W$, the closure of
    $W_0$ is contained inside of a collared open ball or open half-ball depending on
    whether $W\in\mathfrak B$ or $W\in \mathfrak H$, respectively.
\end{enumerate}

Observe that the components of the sets $U_i$ and $V_i$ need not themselves be balls or
half-balls, only have their closures be contained inside of balls or half-balls. Since
the collections $\mathfrak B$ and $\mathfrak H$ have bounded cardinality depending only
on $d$, the parameter space of choices for $(\mathfrak B,\mathfrak H)$ is itself
parameter-free definable.

The number of charts required in the atlas is the only part of the proof which depends
on the dimension of $M$. All other dependencies on dimension fundamentally arise from the
number of charts in the atlas.

{\bf Initializing a scratchpad.}
A schematic illustration of the initialized scratchpad is given in Figure~\ref{fig:iteration-inside-M}.
Fix regular open sets $W$ and $W^{\partial}$ with the following
properties:
\begin{enumerate}
    \item The closure of $W$ is contained in a collared open ball in $M$.
    \item For all components $W_0$ of $W^{\partial}$, the closure of $W_0$ is contained in a collared
    open half-ball in $M$.
    \item If $\hat W$ is an arbitrary regular open set whose
    closure is contained in a collared open half-ball in $M$, then there exists an element
    $g\in\mH$ such that $g(\hat W)\subseteq W^{\partial}$.
    \item Each component of $\partial M$ meets at most one component of $W^{\partial}$.
\end{enumerate}
It is straightforward to see that, in view of Theorem~\ref{thm:kkdlng},
the conditions defining $W$ and $W^{\partial}$ are expressible, and
that such $W$ and $W^{\partial}$ always exist. Next, choose elements $\{g_U\mid U\in\mathfrak B\}$
and $\{g_V\mid V\in\mathfrak H\}$ such that:
\begin{enumerate}
    \item For all $U\in \mathfrak B$ and $V\in\mathfrak H$, we have $g_U(U)$ has compact
    closure inside of $W$ and
    $g_V(V)$ has compact closure inside of $W^{\partial}$.
    \item For distinct $U_1, U_2\in\mathfrak B$, the images $g_{U_1}(U_1)$ and $g_{U_2}(U_2)$
    are disjoint; we place the same requirement on
    distinct elements of $\mathfrak H$. Let \[U_0=\bigcup_{U\in\mathfrak B} g_U(U)\] and
    \[V_0=\bigcup_{V\in\mathfrak H} g_V(V).\]
    \item Choose elements $g_0\in \mH[W]$ and $g_0^{\partial}\in\mH[W^{\partial}]$ such
    that for all distinct $i,j\geq 0$, we have \[g^i_0(U_0)\cap g^j_0(U_0)=\varnothing,\]
    and similarly \[(g^{\partial}_0)^i(V_0)\cap (g^{\partial}_0)^j(V_0)=\varnothing.\]
    Here, we are implicitly using the fact that we may quantify over the arguments of
    the (definable) exponentiation function.
\end{enumerate}

We write $U^i=g_0^i(U_0)$ and $V^i=(g^{\partial}_0)^i(V_0)$, respectively. The reader may
observe that this is the point where the argument fails for $M=I$, since in the case
of the interval the homeomorphism $g^{\partial}_0$ may not exist.

\begin{figure}\label{fig:iteration-inside-M}
  \centering
\begin{tikzpicture}[scale=1.2]

  \draw[thick] (0,0.7) rectangle (4,6);
  \node at (4.3,6) {$W$};

  \draw[fill=gray!20] (2,5.5) circle (0.25);
  \node at (3,5.5) {$g_U(U)$};

  \draw[->, thick] (2,4.9) -- (2,4.1);
  \node[right] at (2.05,4.5) {$g_0$};

  \draw[fill=gray!20] (2,3.5) circle (0.25);
  \node at (3.2,3.5) {$g_0(g_U(U))$};

  \draw[->, thick] (2,2.9) -- (2,2.1);
  \node[right] at (2.05,2.5) {$g_0$};

  \node at (2,1.5) {$\vdots$};

\end{tikzpicture}
 \caption{A schematic of the scratchpad; here we draw the image of one chart $U$ in the
 atlas (which need not actually be a disk) in $W$, and the image under $g_0$. The iterates
 under $g_0$ continue to infinity.}
  
\end{figure}

{\bf Encoding countable sequences of points in $M$.}
For a schematic of this part, see Figure~\ref{fig:nested-annuli-color}.
We now choose a regular open set $P$, which together with the
scratchpad will encode a countable sequence of points in $M$. Here,
we require $P$ to satisfy the following conditions:
\begin{enumerate}
    \item The set $P$ is contained in $\bigcup_i U^i\cup\bigcup_i V^i$.
    This can be expressed by
    requiring for each component of $P$, there is an $i$ so that the $-i^{th}$ power of
    the relevant $g_0$ or $g^{\partial}_0$ is contained in $U_0$ or $V_0$, respectively.
    \item For each $i$, exactly one of the intersections $P\cap U^i$ and $P\cap V^i$ is
    nonempty and isolates 
    a unique point $p_i$ in $U^i$ or $V^i$. From here on,
    write $q_i$ for the backwards image of $p_i$ under
    the $i^{th}$ power of $g_0$ or $g^{\partial}_0$ respectively, followed by the relevant
    $g_U^{-1}$ or $g_V^{-1}$.
\end{enumerate}

Via the set $P$, we have thus encoded (with parameters),
in an unambiguous way, a countably
infinite sequence of points $\{q_i\}_{i\in\N}\subseteq M$. This is the sort
$\on{seq}(M)$.

\begin{figure}[ht]
  \centering
  \begin{tikzpicture}[scale=1]

    \draw[thick] (0,0) circle (1.5);
    \node at (0,-2) {$U^i$};

    \fill[gray!30] (-0.5,0) circle (0.45);
    \fill[blue!50]  (-0.5,0) circle (0.3);
    \fill[red!70]   (-0.5,0) circle (0.15);
    \node at (0.49,0) {$P\cap U^i$};

    \draw[thick] (4,0) circle (1.5);
    \node at (4,-2) {$U^i$};


  \end{tikzpicture}
  \caption{A schematic of two components in $U^i$. The sets $P$ meets $U^i$ and isolates
  a unique point in it.}
  \label{fig:nested-annuli-color}
\end{figure}

Since we can quantify over the arguments in the exponentiation function,
it is straightforward to see from the construction that the membership predicate
$p\in\sigma$ and $\sigma(i)=p$ are both definable, \emph{a priori} with parameters.

{\bf Eliminating parameters.}
It is clear from the descriptions of the regular open sets chosen in the covers and
the relevant homeomorphisms of $\mH$ that are chosen, that the choices are made over
definable sets of parameters. Given two choices of parameters, we simply declare
two interpretations of two sequences of points to be equivalent if for each $i\in\N$,
the $i^{th}$ terms of the sequences represent the same point of $M$; this is possible
in view of Item~\ref{c:points} of Theorem~\ref{thm:kkdlng}. This completes
the proof of the lemma.
\end{proof}

We can now give a modified proof of Lemma~\ref{lem:point-seq} for the interval. Technically
we will only interpret sequences of points in the interior $(0,1)$ of $I$, which
is all that will be needed. It is not
difficult to add ``dummy entries" of two varieties to sequences which stand for possible
choices of endpoints of $I$.

\begin{proof}[Proof of Lemma~\ref{lem:point-seq} for $M=I$]
    We begin by defining
the set of homeomorphisms of $I$ which attract to a point in the interior $(0,1)$ of $I$.
Fixing a point $p_0\in (0,1)$, we may define the set of elements $f\in\mH$ such that for
all $U$ containing $p_0$ and with closure contained in $(0,1)$, and for all $q\in (0,1)$,
there exists an $n\in\N$ such that $f^n(q)\in U_0$. Call these elements of $\mH$ the
\emph{$p_0$--attracting} homeomorphisms.
In light of Theorem~\ref{thm:kkdlng}, the $p$--attracting homeomorphisms of $I$ are
definable, with the point $p_0$ as the sole parameter.

Now, let $f\in\mH$ be a $p_0$--attracting homeomorphism for some $p_0\in (0,1)$, let
$U\subseteq (0,1)$ be a regular open set whose closure is contained in $(0,1)$, let
$U_0\subseteq U$ be a regular open set containing $p$
whose closure is contained in $U$, and let
$g\in \mH[U]$ have the property that for all distinct $i,j\in\N$, we have
$g^i(U_0)\cap g^j(U_0)=\varnothing$. Write $U_i=g^i(U_0)$ for $i\in\N$. Up to now,
we have carried out the interval analogue of initializing the scratchpad.

We now interpret countable sequences of points in $(0,1)$.
We do this by choosing a regular open set $P$ which
isolates a unique point $p_i$ in each $U_i$.
Defining $q_i=f^{-i}g^{-i}(p_i)$, we have unambiguously interpreted the sequence
$\{q_i\}_{i\in\N}$ inside of $\mH$. Moreover, every sequence of points in $(0,1)$ arises
as some such $\{q_i\}_{i\in\N}$, for various choices of $P$ and $f$.
This defines sequences of points in $(0,1)$ with parameters. We declare two sequences
$\sigma_1$ and $\sigma_2$, with different choices of parameters,
to be equivalent if for all $i\in\N$ the encoded points $\sigma_1(i)$ and $\sigma_2(i)$
represent the same point of $(0,1)$.
\end{proof}

{\bf Interpreting pre-graphs}
Armed with the interpretation of sort $\on{seq}(M)$, we can interpret the sort
of \emph{pre-graphs}; we define pre-graphs to be countable subsets $\gam\subseteq M\times M$
such that the projection of $\gam$ to each factor is dense in $M$.

\begin{lemma}\label{lem:pregraph}
    The sort of pre-graphs is uniformly interpretable for manifolds in dimension $d$,
    from the sort $\on{seq}(M)$. Moreover, the predicate
    $(x,y)\in\gam$ expressing that a pair $(x,y)\in M\times M$ is an element of $\gam$
    is parameter-free interpretable.
\end{lemma}
\begin{proof}
    We may quantify over terms of a sequence $\sigma\in\on{seq}(M)$ and thus encode a
    countable subset $\gam$ of $M\times M$ from $\sigma$ by declaring $(x,y)\in\gam$ if
    and only if there exists an $n\in\N$ such that $\sigma(2n)=x$ and $\sigma(2n+1)=y$.
    Density of the projections is expressed by saying that for each nonempty regular
    $U\in \ro(M)$, there is an odd index $i$ and an even index $j$ such that
    $\sigma(i),\sigma(j)\in U$. The set of $\gam$ encoded by this definable set of
    sequences clearly coincides with pre-graphs. We finally put an equivalence relation
    on elements of $\on{seq}(M)$
    encoding pre-graphs, which expresses that $\sigma_1$ and $\sigma_2$
    are equivalent if and only if they encode pre-graphs that are equal as subsets
    of $M\times M$;
    this is evidently a definable equivalence relation. This completes the
    parameter-free interpretation.
\end{proof}

{\bf From pre-graphs to graphs.}

We now pass to graphs of homeomorphisms of $M$.

\begin{lemma}
    Pre-graphs in dimension $d$ admit a parameter-free interpretation of $\HS_0(M)$.
\end{lemma}
\begin{proof}
    We put definable conditions on pre-graphs to guarantee that they define graphs of
homeomorphisms of $M$. Since $M$ is compact, it suffices to require that a pre-graph
$\gam$ extend continuously to the graph of a continuous self-map of $M$ which is injective
and surjective.

\begin{enumerate}
    \item
    {\bf Continuity:} we need only require for all $(x_0,y_0)\in\gam$ that for all open $V$
    containing $y_0$, there is a $U$ containing $x_0$ such that for all $(x,y)\in\gam$
    with $x\in U$, we have $y\in V$. This is clearly expressible. Any $\gam$ satisfying
    this continuity requirement automatically encodes a continuous map
    \[f_{\gam}\colon M\longrightarrow M.\]
    \item {\bf Injectivity:} we need only require that for all disjoint open $U_1$ and $U_2$
    there exist disjoint open $V_1$ and $V_2$ such that if $(x_i,y_i)\in\gam$ for
    $i\in\{1,2\}$ with $x_i\in U_i$, then $y_i\in V_i$.
    \item {\bf Surjectivity:} we need only require that the image of $f_{\gam}$ be dense
    in $M$. This can be achieved by requiring for all nonempty $V$ that there be an
    $(x,y)\in\gam$ with $y\in V$.
\end{enumerate}

Any pre-graph $\gam$ satisfying the foregoing conditions will automatically encode the
graph of a homeomorphism of $M$. Moreover,
every homeomorphism of $M$ is encoded by some pre-graph, simply by taking a dense subset
of the graph of the homeomorphism.
To complete the interpretation of $\HS_0(M)$, we put an equivalence relation on
pre-graphs which expresses that two pre-graphs $\gam_1$ and $\gam_2$
are equivalent if they encode the same
homeomorphism of $M$. For this, it suffices to require that if $(x_1,y_1)\in\gam_1$ with
$x_1\in U$ and $y_1\in V$ then there exists a pair $(x_2,y_2)\in\gam_2$ with
$x_2\in U$ and $y_2\in V$.
\end{proof}

\subsection{Interpreting $\mH$ within $\HS_0(M)$}
Recall that the initial given data is $\mH$, whereas here we have interpreted elements of $\Homeo(M)$ via their graphs; \emph{a priori}, $\Homeo(M)$ may be substantially larger
than $\mH$. We note that it is
straightforward to interpret $\mH$ as a set
within $\HS_0(M)$: indeed, consider the association $g\mapsto \Gamma_g$, which sends an
element $g\in \Homeo(M)$ to the graph of $g$ as a homeomorphism of $M$. We have
$\Gamma_g$ corresponds to a graph of an element of $\mH$ if and only
\[(\exists \gamma)[\forall x\forall y ((x,y)\in\gam_g\leftrightarrow \gamma(x)=y)].\]

Thus, we are justified in saying that $\mH$ can interpret its own elements via graphs, and we are justified in saying we have
interpreted elements of $\mH$ inside of $\HS_0(M)$. We will interpret the group operation
below. We summarize with the following corollary:

\begin{cor}\label{cor:define-graph}
    There is a definable predicate $R\subseteq \mH\times\HS_0(M)$ defining the pairs
    $(g,\gam)$ such that $\gam=\gam_g$ encodes the graph of $g$.
\end{cor}

\subsection{Interpreting the sorts $\HS_n(M)$ for $n\geq 1$}

The interpretation of the sorts $\HS_n(M)$ for $n\geq 1$ is now straightforward, because
of the existence of a computable bijection $\phi\colon \N^2\longrightarrow\N$.

\begin{lemma}\label{lem:hs1}
    For all $n\geq 1$, the sort $\HS_n(M)$ is parameter-free interpretable in $\on{seq}(M)$, uniformly interpretable for manifolds of dimension $d$.
\end{lemma}
\begin{proof}
    We proceed by induction, $\HS_0(M)$ having been interpreted already. To interpret
    $\HS_{n+1}(M)$ once $\HS_n(M)$ has been parameter-free interpreted,
    we use the bijection $\phi$ as above to definably pass from $\N$--indexed
    sequences of points to $\N^2$--indexed points
    $\{q_{(i,j)}\}_{(i,j)\in\N^2}$. For $i$ fixed, we simply require that the (obviously
    parameter-free definable) subsequence
    $\{q_{(i,j)}\}_{j\in\N}$ encode an element of $\HS_n(M)$. It is clear that this
    furnishes a parameter-free interpretation of $\HS_{n+1}(M)$.
    
    It is clear that
    the predicate $\in_n\subseteq\HS_{n}(M)\times\HS_{n+1}(M)$ is parameter-free definable,
    as is the predicate defining the $i^{th}$ term in a sequence in $\HS_n(M)$.
\end{proof}

\subsection{Predicates for manipulating $\HS(M)$}
Most predicates for manipulating sequences in $\HS_n(M)$ are easily seen to be
interpretable, as follows from the fact that one can freely quantify over the arguments
in the exponentiation function;
we have argued concerning membership $\in_n$ and the predicate $s(i)=t$ for
$s\in\HS_{n+1}(M)$ and $t\in\HS_n(M)$ already.

Let $f_1,f_2,f_3\in\HS_0(M)$ be terms in a sequence $\sigma\in\HS_1(M)$. It is easy to see that there is
a predicate expressing that $f_1*f_2=f_3$ in $\Homeo(M)$. Indeed, let $\gam_i$ be graphs
of $f_i$ for $i\in\{1,2,3\}$.
To express that $f_1*f_2=f_3$, it suffices to
express that for all $(x,z)\in\gam_3$ and all open sets $U$ and $V$
such that $x\in U$ and $z\in V$, whenever $(x',y)\in \gam_1$ with $x'\in U$
and all open $W$ such that
$y\in W$, there exists a $(y',z')\in\gam_2$ such that $y'\in W$ and $z'\in V$.

For homeomorphisms of $M$, extended supports are regular open sets which are interpretable
via Rubin's Interpretability Theorem, and which is given by
a purely first order group theoretic formula; see~\cite{Rubin1989}, and 
specifically Theorem 3.6.3 of~\cite{KK21-book} and Section 3.2 of~\cite{dlNKK22}. It is
clear then that we may interpret a new sort which represents
the extended support of an element $f\in\HS_0(M)$, and which is canonically identified with
the extended support of the homeomorphism $f$.
This completes the proof of Theorem~\ref{thm:hered-seq-homeo}.

We have the following consequences of interpreting the sort $\HS_1(M)$ and the
preceding predicates.
\begin{cor}\label{cor:subgroup-define}
    The following conclusions hold.
    \begin{enumerate}
        \item The set of sequences $s\in\HS_1(M)$ which,
        via the identification of $\HS_0(M)$ with $\Homeo(M)$,
        form subgroups of $\Homeo(M)$
        is parameter-free definable.
        \item If $X\subseteq\HS_0(M)=\Homeo(M)$ is arbitrary, then there is a
        predicate \[\on{member}_X\subseteq\Homeo(M),\] using $X$ as a parameter,
        which expresses whether an arbitrary $f\in\Homeo(M)$
        is a finite product of elements of $X$. In particular, if $X$ is parameter-free
        definable then $\on{member}_X$ is parameter-free definable.
    \end{enumerate}
\end{cor}
\begin{proof}
    The first part reduces to requiring for all $f,g\in s$, we have $f^{-1}\in s$ and
    $f\cdot g\in s$. The second part reduces to the existence of a sequence
    $s\in\HS_1(M)$ with $s(0)=1$, with $s(n)=f$ for some $n\in\N$, and such that for
    all $0<m\leq n$ we have $s(n-1)^{-1}s(n)\in X$.
\end{proof}

\section{Intermediate subgroups, mapping class groups, and Theorem~\ref{thm:gp-thy-conseq}}\label{sec:mcg}

We now use the interpretation of the sorts
$\HS(M)$ to extract group-theoretic consequences.
Observe first that $\mH$ interprets $\Homeo(M)$. Indeed, this is part of the content of Theorem~\ref{thm:hered-seq-homeo}.
Next, we can interpret $\Homeo_0(M)$. 
The key to interpreting $\Homeo_0(M)$ is the following result, which appears as Corollary 1.3 in~\cite{KE-1971}.

\begin{thm}[Edwards--Kirby]\label{thm:kirby-edwards}
Let $\mathcal U$ be an open cover of a compact manifold $M$. An arbitrary element $g\in\Homeo_0(M)$ admits a \emph{fragmentation}
subordinate to $\mathcal U$. That is, $g$ can be written as a composition of homeomorphisms that are supported in elements of $\mathcal U$.
\end{thm}

\begin{prop}\label{prop:interpret-homeo0}
The group $\mH$ interprets $\Homeo_0(M)\subseteq\HS_0(M)$.
\end{prop}

As always, the interpretation of $\Homeo_0(M)$ in $\mH$ is uniform in manifolds of bounded dimension.

\begin{proof}[Proof of Proposition~\ref{prop:interpret-homeo0}]
It suffices to construct a formula $\on{isotopy}_0(\gamma)$ that is satisfied by a homeomorphism $g$ if and only if $g$ is isotopic to the
identity. We will carry out the construction for closed manifolds, with the general case being similar.

Consider $\Gamma_g$, the graph of a homeomorphism as obtained from interpreting the sort
$\HS_1(M)$, and let $\mathfrak B=\{U_1,\ldots,U_{n(d)}\}$
be a cover of $M$, with each component of each $U_i$ having compact closure inside of a
collared open ball.

By imposing suitable definable conditions on the data defining $\Gamma_g$, we may insist that there exists an $i$ and a component $\hat U_i$ of $U_i$ such
for all $(p,q)\in\gam_g$, we have $p=q$ unless $p\in\hat U_i$. Specifically, we may write
\[\on{small-sup}(\gam):=(\forall(x,y)\in\gam)(\exists i\leq n(d))(\exists \hat u\in\pi_0
(u_i))[x\notin \hat u\rightarrow x=y];\] in this formula we are implicitly treating elements
of $\mathfrak B$ as parameters.

This condition implies that the homeomorphism $g$ encoded by $\gam$
is the identity outside of $\hat U_i$. Since
$\hat U_i$ is compactly contained in the interior of a collared ball in $M$
we have that $g$ is isotopic to the identity, as follows from the Alexander trick.

By quantifying over all such covers $\mathfrak B$ of $M$, we
thus obtain a parameter-free definable set $X\subseteq\HS_0(M)$ consisting
of graphs of elements of $\Homeo(M)$ which satisfy $\on{small-sup}$ for some such cover.

By Theorem~\ref{thm:kirby-edwards}, we have that $g\in\Homeo(M)$ is isotopic to the identity if and only if $g$ is a product of
a finite tuple of homeomorphisms lying in $X$. By Corollary~\ref{cor:subgroup-define},
it follows that $\Homeo_0(M)$ is parameter--free definable as a subset of the sort
$\HS_0(M)$.
\end{proof}

An arbitrary subgroup $\Homeo_0(M)\leq\mH'\leq\Homeo(M)$ is automatically of countable
index in $\Homeo(M)$, as follows from the fact that
for a compact manifold, $\Homeo(M)$ is separable
and therefore has countably many connected components.

\begin{proof}[Proof of Proposition~\ref{prop:unif-intermediate}]
A subgroup \[\Homeo_0(M)\leq \mH'\leq \Homeo(M)\] can be encoded by a definable equivalence class of countable subsets of 
$\Homeo(M)$; indeed, if $\underline g$ is a sequence
then we obtain a subgroup $\mH_{\underline g}$
(viewed as a subset of $\HS_0(M)$) via 
\[\mH_{\underline g}=\{h\mid (\exists g\in\underline g) [h\in g\cdot\Homeo_0(M)]\},\]
after adding the further condition that $\mH_{\underline g}$ be a group (which can be guaranteed by
imposing the first order condition that $\underline g$ be a group,
for instance). Two
sequences of homeomorphisms $\underline g$ and $\underline h$ are equivalent if $\mH_{\underline g}=\mH_{\underline h}$. Since
the mapping class group of $M$ is countable, any such subgroup $\mH'$ occurs as $\mH_{\underline g}$ for some sequence $\underline g$.
We thus obtain a canonical bijection between subgroups $\mH'$ as above and
suitable equivalence classes of sequences of homeomorphisms, as desired.

We have already shown that $\Homeo(M)$ and $\Homeo_0(M)$ are interpretable without parameters. The group $\mH$ itself is also
definable without parameters in the interpretation of $\Homeo(M)=\HS_0(M)$ in $\mH$, as is part of the content of
Theorem~\ref{thm:hered-seq-homeo}.
\end{proof}

It is not difficult to argue the conclusions of Theorem~\ref{thm:gp-thy-conseq}, and so
we only sketch the arguments. Because
$\Homeo(M)$ and $\Homeo_0(M)$ are parameter-free interpretable in $\mH$, so is $\Mod(M)$.
The sorts of countable subgroups of $\Homeo(M)$ and $\Mod(M)$
are parameter-free interpretable, by
Corollary~\ref{cor:subgroup-define}; it is immediate that one can quantify over arbitrary
subsets of countable subgroups, since these subsets will always be countable. All of
the countable algebra of groups can be formalized within $\mathrm{ACA}$ or slightly
stronger systems, which is
substantially weaker than full second order theory of countable groups to which we have
access: see~\cite{simpson-book}, page 14, and also Chapter III.
Here and for
the rest of the section, ``subgroup" will refer to a subgroup of $\Homeo(M)$ or of $\Mod(M)$.

Membership
in a fixed countable subgroup follows from Corollary~\ref{cor:subgroup-define}. Finite
generation asks whether for a countable subgroup, there exists a sequence $\sigma$
wherein all but
finitely many terms are the identity, so that every element in the subgroup can be written
as a finite product of entries in $\sigma$; this is clearly expressible: indeed, we think of
the
set $X$ in Corollary~\ref{cor:subgroup-define} as a sequence $\sigma$ where there
exists an $n\in\N$ such that for all $i>n$ the term $\sigma(i)$ is the identity. Finite
presentability is slightly more complicated but still straightforward.

Finite index subgroups of a given countable subgroup are easily defined, using the subgroup
itself as a parameter; thus, residual finiteness is expressible. For finitely
generated groups, linearity can be expressed via the
Lubotzky Linearity Criterion~\cite{Lubotzky1988}, and in general by quantifying suitably over
countable subgroups of $\mathrm{GL}_n(\bC)$; we omit the tedious details.

Isomorphism of a countable subgroup with a particular definable group is straightforward.
Amenability can be encoded with the F{\o}lner Criterion; see Chapter 2 of~\cite{gelander-cohen}.
Kazhdan's Property (T) for
finitely generated subgroups can be encoded using Ozawa's Criterion, which is the main
result of~\cite{Ozawa2016}.

\section{Descriptive set theory in $\Homeo(M)$}\label{sec:descriptive}

In this section, we show how to interpret the projective hierarchy in $\Homeo(M)$ and
characterize definability via Theorem~\ref{thm:definable}. The interpretation is modeled on the fact that
the projective hierarchy in Euclidean space is definable in second order arithmetic. 

In this section, we emphasize that all interpretations are \emph{uniform}; that is,
there is a single formula, depending only on the dimension of the manifold $M$,
which defines all open sets of $\Homeo(M)$ (as subsets of $\HS_0(M)$) for various choices
of parameters. The same holds for all sets at the various levels of the Borel hierarchy,
analytic sets, and sets in the projective hierarchy.

\subsection{Generalities on descriptive set theory}
The reader is
directed to~\cite{kechris-book,moschovakis-book} for a more thorough background.
Suppose that we are given a Polish (i.e.~completely metrizable and separable) space $X$.
We recall the definition of the Borel hierarchy.
  For every nonempty countable ordinal $\alpha$ one can define the families $\bSigma^{0}_{\alpha}(X)$ and $\bPi^{0}_{\alpha}(X)$ of subsets of $X$ as follows. 
  \begin{itemize}
	  \item The class $\bSigma_1^0(X)$ consists of all open subsets of $X$.
	  \item For all $\alpha<\aleph_{1}$ the class of $\bPi^{0}_{\alpha}$ is the collection of complements of subsets in $\bSigma^{0}_{\alpha}(X)$.
    \item For all limit ordinal $\alpha<\aleph_{1}$ we have $\Sigma^{0}_{\alpha}(X)=\bigcup_{\beta<\alpha}\Sigma^{0}_{\beta}(X)$.
    \item For all $\alpha<\aleph_{1}$ the family $\bSigma^{0}_{\alpha+1}(X)$ consists of all countable unions of sets in $\bPi_{\alpha}^{0}(X)$.
  \end{itemize}
 A subset of $X$ is called \emph{Borel} if it belongs to $\Sigma^{0}_{\alpha}$ for some $\alpha<\aleph_{1}$. 

A subset of $X$ is \emph{analytic} if it is a continuous image of Baire space $\mN=\N^{\N}$;
equivalently, a subset $A\subseteq X$ is analytic if and only if there is a closed subset of
$C\subseteq X\times \mN$ such that $A$ is the projection of $C$ to $X$.
Observe that Baire space and its topology
are parameter-free interpretable in second order arithmetic.

In the projective hierarchy, analytic sets are called
$\bSigma_1^1$. One can extend the notation of the Borel hierarchy
in order to define classes $\bSigma^{1}_{\alpha}(X)$ and $\bPi^{1}_{\alpha}(X)$ 
for all $\alpha<\aleph_{1}$;
here we will need only concern ourselves with integer values of $\alpha$, for which the definition can be given by usual induction as follows:
\begin{itemize}
	\item For all $n$ for which $\bSigma^{1}_{n}(X)$ has been defined,
    let $\bPi^{1}_{n}(X)$ be the class of complements of sets in $\bSigma^{1}_{n}$. 
	      In particular, for $n=1$, we obtain the family $\bPi^{1}_{1}(X)$ of projective sets in $X$.  
	\item A set $Z\subseteq X$ is in $\bSigma_{n+1}^1$ if there is a $\bPi_n^1$ subset
    $Y\subseteq X\times\mN$ such that $Z$ is the projection of $Y$ to $X$. 
\end{itemize}   

The sets $\{\bSigma_n^1(X)\}_{n\geq 1}$ form the \emph{projective
hierarchy}; when $X$ is a Euclidean space, the projective hierarchy
is easily seen to be definable in second order arithmetic. We note that in the definition of the projective hierarchy, the factor
$\mN$ can be replaced by an arbitrary uncountable Polish space (e.g.~$\Homeo(M)$ itself).

\newcommand{\bo}[0]{\mathrm{B}}
\newcommand{\ord}[0]{\aleph_{1}}
\newcommand{\ph}[0]{\mathrm{Pr}}

It is a standard fact that a subset $X\subseteq\R$ is definable (with parameters)
in second order arithmetic
if and only if it lies in the projective hierarchy. Theorem~\ref{thm:definable} establishes
the corresponding characterization of definable subsets of $\Homeo(M)$. We prove
one direction first.

\begin{prop}\label{prop:define-one-direction}
    Let $X\subseteq\Homeo(M)$ be definable with parameters, as a subset of the sort
    $\HS_0(M)$.
    Then $X$ lies in the projective hierarchy. If $X\subseteq\mH$ is definable with
    parameters in the language of group theory, then $X$ lies in the projective hierarchy.
\end{prop}
\begin{proof}
    This follows by induction on the quantifier complexity $\phi$ of a formula defining $X$.
    Equalities and inequalities (with parameters) in $\Homeo(M)$ define closed and open sets
    respectively, and so if $X$ is defined by quantifier-free formula then it is certainly
    Borel.

    Suppose that $X$ is defined by \[\phi(x)=(\exists y)\psi(x,y,a),\] where $\psi$ defines
    a set \[Y\subseteq(\Homeo(M))^k\] in the projective hierarchy for some $k\in\N$, and
    where
    $a$ is a tuple of parameters. Then $X$ is given
    by a projection of $Y$ to a smaller Cartesian power of copies of $\Homeo(M)$, and so
    $X$ lies at most one level higher than $Y$ in the projective hierarchy. If
    $X\subseteq\mH$ then a canonical definable identification between $X$ and a subset of
    the sort $\HS_0(M)$ is given by Corollary~\ref{cor:define-graph}.
    The proposition
    now follows easily.
\end{proof}

In the remainder of this section, we will interpret sorts for levels of the projective
hierarchy together with the predicate $\in$,
and show that every subset of $\Homeo(M)$ in the projective hierarchy is definable
with parameters (as a subset of the sort $\HS_0(M)$), and every subset of the home sort
$\mH$ in the projective hierarchy is definable with parameters.

\subsection{Open sets}\label{ss:open}
We interpret
the compact-open topology on $\Homeo(M)$ directly.
First, cover $M$ by regular open sets.
Since regular open sets themselves are encoded by
definable equivalence classes of homeomorphisms by associating their extended supports (cf.~Theorem~\ref{thm:kkdlng}),
finite covers of $M$ can be encoded by
equivalence classes of finite tuples of homeomorphisms. To define finite tuples $\tau$
of homeomorphisms whose
supports cover $M$ since one need only express that for all $p\in M$, there exists an
$f\in\tau$ such that $p\in\supp^e(f)$.

From a finite cover $\mathcal V$ of $M$,
one can define an open set $U_{\mathcal V}\subseteq\Homeo(M)$ by considering homeomorphisms $f$ such that for all $p\in M$,
there exists a $V\in\mathcal V$ such that both $p$ and $f(p)$ lie in $V$.
It is not so difficult to see that $U_{\mathcal V}$ is indeed
open.

Now, if $f\in\Homeo(M)$ and $\mathcal V$ is a finite covering of $M$ then we set
$U_{\mathcal V}(f)$ to be the set of homeomorphisms $g$ such that $g^{-1}f$ lies in $U_{\mathcal V}$ as defined above. Observe that
for a given cover $\mathcal V$ and $f\in\Homeo(M)$, the subset of $\HS_0(M)$ contained $U_{\mathcal V}(f)$ is definable (with $\mathcal V$
as a parameter).

As $\mathcal V$ varies over finite covers of $M$ and $f$ varies over $\Homeo(M)$, we have that the sets  $U_{\mathcal V}(f)$ form a basis
for the compact--open topology of $\Homeo(M)$.
Thus, a basis for the topology on $\Homeo(M)$ is encoded by certain equivalence classes of finite tuples of
elements of $\Homeo(M)$, which is to say certain definable subsets of $\HS_1(M)$, up
to definable equivalence.

More explicitly,
for a tuple $(f_1,\ldots,f_n)$, we definably associate extended
supports via
$f_i\mapsto V_i=\supp^e f_i$ for $i\geq 2$, 
while requiring that \[M\subseteq\bigcup_{i=2}
^n\on{supp}^e(f_i).\] Setting $\mathcal V=\{V_2,\ldots,V_n\}$, such a tuple of homeomorphisms encodes the set $U_{\mathcal V}(f_1)$. This interpretation is clearly
uniform in the sense described at the beginning of the section.
The predicate $\in$ is trivial to interpret.

We see now that basic open sets are
interpretable as a definable subset of $\HS_1(M)$,
up to definable equivalence by setting two tuples to be equivalent if and
only if the basic open sets they encode contain the same homeomorphisms (viewed as elements
in the sort $\HS_0(M)$).

An arbitrary open set is then interpreted as a countable sequence of basic open sets, with a homeomorphism $f$ being a member of
the open set if and only if it is a member of one of the elements in the sequence. Since basic open sets are parameter--free interpretable
in $\HS_1(M)$, we see that open sets are parameter--free interpretable in $\HS_2(M)$. Closed sets are then simply complements of open
sets. It is trivial to interpret the membership relation of homeomorphisms in an open
or closed set.

\begin{cor}\label{cor:open-definable}
    The sorts of open and closed sets of $\Homeo(M)$, viewed as subsets of $\HS_0(M)$,
    are uniformly interpretable with parameters. Open and closed subsets of $\mH$ are
    uniformly definable with parameters.
\end{cor}
\begin{proof}
    We have argued that open sets in $\Homeo(M)$ are encoded by elements in $\HS_2(M)$.
    Thus, a particular open set $U$ is identified with a parameter-free definable
    equivalence class of elements $\tau\in \HS_2(M)$, and $f\in\HS_0(M)$ if and only if
    there exists a $\sigma\in\tau$ such that $f\in\sigma$. The case of closed sets is
    identical. The definability of open and closed sets in the home sort follows now
    from Corollary~\ref{cor:define-graph}.
\end{proof}

With a minor variation on the preceding arguments, we can recover the topology on $\Homeo(M)^{\ell}\times\mN^k$ for all $
\ell,k\geq 0$.
Note that for $k\geq 1$,
we have $\Homeo(M)^{\ell}\times\mN^k\cong \Homeo(M)\times\mN$.
We record the following corollary.

\begin{cor}\label{cor:open-closed}
The sorts of open and closed subsets of $\Homeo(M)^{\ell}\times\mN^k$
are parameter--free interpretable
in $\mH$, and open and closed sets in these spaces are uniformly interpretable
with parameters.
\end{cor}

\subsection{The Borel hierarchy}

  \newcommand{\subg}[1]{\langle #1 \rangle}
  \newcommand{\seq}[0]{\N^{<\N}} 

The Borel hierarchy of $\Homeo(M)$ and $\mH$ is now straightforward to interpret. We first
indicate an interpretation of finite levels of the Borel hierarchy, followed by
the case of arbitrary
countable ordinals using Borel codes; see Section 1.4 of~\cite{gao-book}.

We have already interpreted open and closed sets in $\HS_2(M)$, which corresponds to
$\bSigma^0_1(\Homeo(M))$ and $\bPi^0_1(\Homeo(M))$, respectively; we will suppress the
notation of $\Homeo(M)$ since it will not cause confusion.

By induction,
$\bSigma^0_k$ and $\bPi^0_k$ are uniformly interpreted in $\HS_{k+1}(M)$.
By definition, elements of $\bSigma^0_{k+1}$ are countable unions of elements of
$\bPi^0_k$, which are then encoded by definable equivalence classes of elements
in $\HS_{k+2}(M)$. Elements of $\bPi^0_{k+1}$ are just given by
complementation. The proof of the following is nearly identical to that of
Corollary~\ref{cor:open-definable}

\begin{cor}\label{cor:borel-definable}
    Let $X\subseteq\Homeo(M)$ lie in a finite level of the Borel hierarchy.
    Then $X$ is uniformly interpretable with
    parameters, viewed as a subset of $\HS_0(M)$. If $X\subseteq\mH$ lies in a finite level
    of the Borel
    hierarchy then $X$ is uniformly definable with parameters in the home sort.
\end{cor}

For the general Borel hierarchy, it is helpful to use \emph{Borel codes}, which are a
standard tool in descriptive set theory. For a Polish space $X$, one chooses a countable
basis $\{U_{\tau}\}_{\tau\in\N^{<\N}}$
for the topology of $X$. A Borel set $Y\subseteq X$ is encoded
by a \emph{labeled, well-founded tree}, the definition of which we briefly recall here;
cf.~Section 1.4 of~\cite{gao-book}.
A tree $T\subseteq\N^{<\N}$ is a prefix-closed subset, where elements of $\N^{<\N}$ (also
called nodes) are
viewed as finite sequences; there is an obvious notion of length for a node.
A tree $T$ is \emph{well-founded} if there is no infinite sequence
$\{\tau_i\}_{i\in\N}$ where $\tau_{i-1}$ is a prefix of $\tau_i$. An element $\tau\in T$
is \emph{terminal} if it admits no proper extension in $T$.

If $\tau\in T$ then one
writes $T_{\tau}$ for the set of suffixes of elements of $T$ which have $\tau$ as a prefix,
so that $T_{\tau}$ is itself a tree. A well-founded tree $T\neq\varnothing$
together with a label
function $\lambda\colon T\longrightarrow\N$ forms a Borel code provided that:
\begin{enumerate}
    \item If $\tau\in T$ is non-terminal then $\lambda(\tau)\in\{0,1\}$.
    \item If $\tau\in T$ is non-terminal and $\lambda(\tau)=0$ then there exists a unique
    $\sigma\in T$ extending $\tau$ by exactly one entry, i.e.~of length
    exactly one more than $\tau$.
\end{enumerate}

If $\tau\in T$ and $\lambda$ is a labeling of $T$ then there is an obvious labeling of
$T_{\tau}$ which we also call $\lambda$.

The \emph{rank} of $\tau\in T$ is defined recursively:
\begin{enumerate}
    \item If $\tau$ is terminal then the rank of $\tau$ is zero.
    \item If $\tau\in T$ is not terminal, then the rank of $\tau$ is one more than the
    supremum of the ranks of the one-entry extensions of $\tau$ in $T$, i.e.~of length
    exactly one more than $\tau$.
    \item The rank of $T$ is the rank of the empty sequence $\varnothing\in T$.
\end{enumerate}

Choose a bijection $\N^{<\N}$ with $\N$, which we write $\tau\mapsto \langle\tau
\rangle$.
A Borel set $B_{(T,\lambda)}$ in $X$ is encoded by the pair $(T,\lambda)$ as follows.
\begin{enumerate}
    \item If $\varnothing$ is the only node of $T$ then $B_{(T,\lambda)}=U_{\tau}$, where
    $\langle\tau\rangle=\lambda(\varnothing)$.
    \item If $\varnothing$ is non-terminal and $\lambda(\varnothing)=0$ then there is
    a unique node $\sigma$ of length one extending $\varnothing$. We write
    $B_{(T,\lambda)}=X\setminus B_{(T_{\sigma},\lambda)}$.
    \item If $\varnothing$ is non-terminal and $\lambda(\varnothing)=1$, then write
    $\{\sigma_i\}_{i\in\N}$ for the nodes of length one in $T$ and define \[B_{(T,\lambda)}
    =\bigcup_i B_{(T_{\sigma_i},\lambda)}.\]
\end{enumerate}

This encoding makes sense because of the well-foundedness of $T$. A set in $X$ is Borel
if and only if it admits a Borel code. Moreover, for a countable ordinal $\alpha$,
a Borel set lies
in $\bSigma_{\alpha}^0$ if and only if it is encoded by a Borel code encoded by $(T,\lambda)$
of rank at most $\alpha$ with $\lambda(\varnothing)\neq 0$. Similarly,
a Borel set lies
in $\bPi_{\alpha}^0$ if and only if it is encoded by a Borel code encoded by $(T,\lambda)$
of rank at most $\alpha$ with $\lambda(\varnothing)= 0$.

\begin{cor}
    The following are uniformly parameter--free interpretable in $\mH$:
    \begin{enumerate}
        \item The Borel sets $\mathcal B$ of $\Homeo(M)$, viewed as subsets of $\HS_0(M)$.
        \item The membership predicate for Borel subsets.
        \item A rank predicate $\rk\subseteq \mathcal B\times \aleph_1$, consisting of
        pairs $(A,\alpha)$ with $A\in\bSigma_{\alpha}^0$.
    \end{enumerate}
\end{cor}
\begin{proof}
    This is nearly immediate. First, countable ordinals are parameter-free definable
    in second order arithmetic. Moreover, there is a definable bijection between $\N^{<\N}$
    and $\N$, so that in second order arithmetic we may define (without parameters)
    well-founded trees and hence
    Borel codes.

    It is straightforward to see that, in light of Subsection~\ref{ss:open}, we may
    have direct access to countable bases for the topology on $\Homeo(M)$. It is
    similarly straightforward to see that via Borel codes, we may encode:
    \begin{enumerate}
        \item Borel sets.
        \item A parameter-free predicate which expresses when two Borel codes encode the
        same Borel set.
        \item The rank function $\rk$.
        \item The membership predicate in members of the class of Borel sets.
    \end{enumerate}

    Moreover, individual Borel sets are interpretable with parameters. We omit the
    remaining details.
\end{proof}

\subsection{The projective hierarchy}

In this section, we complete the proof of Theorem~\ref{thm:definable}; precisely, we will
show that the levels $\bSigma_n^1$ and $\bPi_n^1$ of the projective hierarchy of
$\Homeo(M)$ are uniformly interpretable sorts, and that a set in the projective hierarchy is
definable with parameters, uniformly within a level of the hierarchy.

By definition, an analytic set in a Polish space $X$ is a continuous image of $\mN$. Equivalently, an analytic set in $X$ is the projection
of a closed subset of $X\times\mN$ to $X$. By Corollary~\ref{cor:open-closed}, we have interpreted closed subsets of $\Homeo(M)\times\mN$. More precisely, an open set in
$\Homeo(M)\times\mN$ is a countable union of basic open sets in the product, which can
be taken to be pairs of basic open sets in each factor. It is not difficult to see then
that open sets in $\Homeo(M)\times\mN$ can be encoded in $\HS_3(M)$, and closed sets
by complementation.
If \[C\subseteq\Homeo(M)\times\mN\] is a closed subset then the set \[Y_C=\{f\mid(\exists x)[(f,x)\in C]\}\] is analytic,
and every analytic set arises this way. Thus, membership of a homeomorphism
$f\in\Homeo(M)$ in an analytic (or co-analytic) set is expressible.

\begin{cor}\label{cor:analytic-definable}
    Let $X\subseteq\Homeo(M)$ be analytic or co-analytic. Then $X$ is uniformly
    interpretable with
    parameters, viewed as a subset of $\HS_0(M)$. If $X\subseteq\mH$ is analytic or
    co-analytic then $X$ is uniformly definable with parameters in the home sort.
\end{cor}

It is trivial to extend this discussion to analytic and co-analytic subsets of finite
Cartesian powers of $\Homeo(M)$.

To interpret the higher levels of the projective hierarchy, suppose
by induction that $X\subseteq
\Homeo(M)^{\ell}$ is a $\bPi_n^1$ set that is definable with parameters. Then the set
$Y\subseteq\Homeo(M)^{\ell-1}$ given by projecting $X$ to the first $\ell-1$ factors is
$\Sigma_{n+1}^1$, and every $\Sigma_{n+1}^1$ occurs this way. Thus, we have:

\begin{cor}\label{cor:projective-definable}
    Let $X\subseteq\Homeo(M)$ lie in a fixed level
    the projective hierarchy. Then $X$ is uniformly interpretable with
    parameters, viewed as a subset of $\HS_0(M)$. If $X\subseteq\mH$ lies in a fixed level
    of the
    projective hierarchy then $X$ is uniformly definable with parameters in the home sort.
\end{cor}

This completes the proof of Theorem~\ref{thm:definable}.

\section{Undefinability of sentences isolating manifolds}\label{sec:undefinable}
Throughout this section,
we will limit ourselves to full homeomorphism groups of manifolds; it is easy to see that the entire
discussion could be carried out for any subgroup between $\Homeo_0$ and $\Homeo$, and we
make this choice for the sake of concision.

In this section, we prove Theorem~\ref{thm:undefinable}, Theorem~\ref{thm:rice}, and
Theorem~\ref{thm:rice-strong}; together, these results 
show that many natural sets of natural numbers associated to homeomorphism groups of
manifolds
are not definable in 
second order arithmetic. 

We fix an arbitrary numbering of the symbols in the language of group theory, and thus obtain a computable G\"odel numbering
of strings of symbols in this language. As is standard, well-formed formulae and sentences are definable in arithmetic, which is to say the set of
G\"odel numberings of formulae and sentences are definable in first order arithmetic. For formulae and sentences $\psi$ in the language
of group theory (and occasionally, by abuse of notation, in arithmetic),
we will write $\#\psi$ for the corresponding G\"odel numbers. For a class of sentences in group theory, the definability
of the set of G\"odel numbers of sentences in that class is independent of the G\"odel numbering used.
The proof of Theorem~\ref{thm:undefinable}
will follow ultimately from Tarski's well-known undefinability of truth~\cite{BellMachover-book,hinman-book,manin-book}. That is, there is no
predicate $\on{True}$ that is definable in arithmetic (first or second order) such that for all sentences $\phi$ in second
order arithmetic, we have
\[\phi\longleftrightarrow \on{True}(\#\phi).\] See Theorem 12.7 of~\cite{jech-book} for a general discussion.

For the remainder of this section, we will fix a uniform interpretation of
second order arithmetic in homeomorphism groups of compact manifolds. If $\psi$ is
an arithmetic sentence, we will write $\tilde\psi$ for the corresponding interpreted
group theoretic statement. Thus, we have $\mathrm{Arith}_2\models\psi$ if and only if
$\Homeo(M)\models\tilde\psi$ for all compact manifolds $M$; here we use $\mathrm{Arith}_2$
to denote second order arithmetic, as opposed to $\N$ which usually denotes first order
arithmetic. For a fixed G\"odel numbering in arithmetic, the association
$\#\psi\mapsto\#\tilde\psi$ is computable.

Let $M$ be a fixed compact manifold and let $\psi$ be a sentence in group theory.
Recall that $\psi$ \emph{isolates} $M$ if for all compact manifolds $N$, we have
\[\Homeo(N)\models\psi\longleftrightarrow M\cong N.\] Notice that if $\psi$ isolates $M$ then
$\Homeo(M)\models \psi$.
Similarly, we will say $\psi$ \emph{isolates a manifold}
if there is a unique compact manifold $M$ such that $\Homeo(M)\models \psi$.

Recall that Rice's Theorem from computability theory asserts that if $\mathcal C$ is a
class of partial recursive functions then the set $\{n\mid \phi_n\in\mC\}$ is computable
if and only if $\mathcal C$ is empty or equal to the whole class of partial recursive
functions. Here, we have adopted the standard notation $\phi_n$ for the $n^{th}$ function
computed by the universal Turing machine. See Corollary 1.6.14 in~\cite{Soare-book}.

Here we will prove two analogues of
Rice's Theorem for homeomorphism groups of manifolds. Let $\mathcal M$ be a class of
(homeomorphism classes) of compact manifolds. We will say that $\mathcal M$ is \emph{finitely
axiomatized} if there is a first order sentence $\phi_{\mathcal M}$ in the language of
group theory such that for all compact manifolds $M$, we have 
\[M\in\mathcal M \Leftrightarrow \Homeo(M)\models \phi_{\mathcal M};\] in particular,
$\phi_{\mathcal M}$ isolates precisely those manifolds $M$ which lie in $\mathcal M$.

\begin{thm}\label{thm:rice}
    Let $\mathcal M$ be a class of compact manifolds that is finitely axiomatized,
    and let \[
    \on{axiom}(\mathcal M):=
    \{\#\phi\mid \phi \textrm{ finitely axiomatizes } \mathcal M\}.\] Then
    $\on{axiom}(\mathcal M)$ is not definable in second order arithmetic.
\end{thm}

The reader may note that Theorem~\ref{thm:rice} implies that even the set of sentences
which are \emph{false} for all compact manifold homeomorphism groups
(i.e.~$\on{axiom}(\varnothing)$) is so complicated
as to be undefinable in second order arithmetic.

Even more generally, let $\mathcal A$ denote the set of all homeomorphism classes of
compact manifolds, and let $\mathcal F$
denote the set of nonempty subsets of
$\mathcal A$ that are finitely axiomatized by first order sentences
in the language of group theory.

\begin{thm}\label{thm:rice-strong}
    Let $A\subseteq\mathcal F$ be nonempty and proper. 
    Then the set \[\chi(A)=
    \{\#\psi\mid \psi\textrm{ finitely axiomatizes some $a\in A$}\}\] is not definable in second
    order arithmetic.
\end{thm}

Before giving the proof of Theorem~\ref{thm:rice-strong}, we note that it implies
Theorem~\ref{thm:rice}, as well as
Theorem~\ref{thm:undefinable} from the introduction.

\begin{proof}[Proof of Theorem~\ref{thm:rice}]
    Suppose first that $A=\mathcal{M}$ is nonempty and finitely axiomatized.
    We have $A\neq\mathcal F$ because $A$ is a subset of $\mathcal A$ and because
    each of the countably infinitely many
    singletons of $\mathcal A$ is finitely axiomatized; this is part of the content
    of Theorem~\ref{thm:kkdlng}.
    By Theorem~\ref{thm:rice-strong}, we have that $\chi(A)=\on{axiom}(\mathcal M)$
    is not definable in second order arithmetic.

    To see that $\on{axiom}(\varnothing)$ is not definable in second order arithmetic,
    we simply note that for all arithmetic sentences $\psi$, we have $\#\tilde\psi\in
    \on{axiom}(\varnothing)$ if and only if $\psi$ is false in $\mathrm{Arith}_2$.
    This violates the undefinability of truth.
\end{proof}

\begin{proof}[Proof of Theorem~\ref{thm:undefinable}]
    Let $M$ be a fixed compact manifold. The undefinability of the set $\on{Sent}_M$ is
    precisely the conclusion of Theorem~\ref{thm:rice} when $\mathcal M=\{M\}$.

    For the undefinability of $\on{Sent}$, we note that if $\phi$ isolates some compact
    manifold $M$ then for all
    arithmetic sentences $\psi$, we have $\phi\wedge\tilde\psi$ isolates some compact
    manifold $M$ if and only $\mathrm{Arith}_2\models\psi$; this is simply because
    $\neg\tilde\psi$ is always false in compact manifolds homeomorphism groups, and
    $\neg\phi$ isolates no compact manifold because there are at least three pairwise
    non-homeomorphic compact manifolds. Thus, if $\on{Sent}_M$ were definable then
    we would be able to define truth in $\mathrm{Arith}_2$, a contradiction.
\end{proof}

To add to the complexity
of the
sets $\on{Sent}_M$ and $\on{Sent}$, note that
 it is well-known that there is a Diophantine equation which
does not admit a solution if and only if ZFC is consistent (or, if and only if PA is consistent); cf.~Chapter 6 of~\cite{MurtyFodden}.
For such an equation, we may express the 
nonexistence of a solution to a particular Diophantine equation as a sentence $\phi$ in first order arithmetic. Interpreting this sentence
in $\Homeo(M)$ to get a group theoretic sentence $\tilde\phi$, we see that if $\psi$ isolates $M$ then
$\psi\wedge\tilde\phi$ isolates $M$ if and only if ZFC is consistent (or, if and only if PA is consistent). 
A similar argument works for sentences isolating some manifold. Thus, for a particular
G\"odel numbering, there are numbers whose membership in $\on{Sent}_M$ and $\on{Sent}$ cannot be proved in ZFC.

We finally establish Theorem~\ref{thm:rice-strong}.

\begin{proof}[Proof of Theorem~\ref{thm:rice-strong}]
    Let $\phi\in\chi(A)$ finitely axiomatize some $a\in A$
    and let $\theta$ finitely axiomatize some $\varnothing\neq b\notin A$; the sentence
    $\theta$ exists since $A$
    is assumed to be proper. For each arithmetic
    sentence $\psi$, we let \[\psi^*:=(\tilde\psi \wedge \phi)\vee (\neg\tilde\psi\wedge
    \theta).\] Notice that $\#\psi^*\in\chi(A)$ if and only if $\mathrm{Arith}_2
    \models \psi$. Indeed, if $\psi$ is true in arithmetic then $\tilde\psi$ is
    true for all compact manifolds and $\neg\tilde\psi\wedge\theta$ is false for all compact manifolds. In this case, $\psi^*$ is true in $\Homeo(M)$ if and only
    if $\phi$ holds in $\Homeo(M)$, in which case $\#\psi^*\in\chi(A)$.

    Conversely, suppose that $\psi$ is false in arithmetic. Then $\tilde\psi \wedge \phi$
    is false for all compact manifolds, and so $\psi^*$ is true for $\Homeo(M)$ if and
    only if $\Homeo(M)\models \theta$, in which case $M\in b\notin A$. It follows that
    $\#\psi^*\notin \chi(A)$.

    Thus, if $\chi(A)$ were definable in second order arithmetic then we could define
    truth, a contradiction.
\end{proof}

Theorem~\ref{thm:rice-strong} has many other consequences regarding undefinability.
As a single example, a finite list of compact manifolds is finitely axiomatized, in view of
Theorem~\ref{thm:kkdlng}; the set of sentences axiomatizing finite collections of manifolds
is itself undefinable.

\section*{Acknowledgements}

The second author is supported by the Samsung Science and Technology Foundation under Project Number SSTF-BA1301-51 and
 by KIAS Individual Grant MG084001 at Korea Institute for Advanced Study. The first 
 author was partially supported by NSF Grant DMS-2002596,
 and is partially supported by NSF Grant
 DMS-2349814. The authors thank 
 M.~Brin, J.~Hanson, O.~Kharlampovich, and C.~Rosendal and for helpful discussions.
 The authors also thank anonymous referees for many comments which improved the paper.

  \bibliographystyle{amsplain}
  \bibliography{ref}
  
\end{document}